\documentclass[a4paper,10pt]{article}
\usepackage{amsmath,amsthm}
\usepackage{amssymb}

\usepackage{amsfonts}
\usepackage{amsbsy}
\usepackage{mathtools} 

\usepackage{xcolor}

\usepackage[top=1in, bottom=1in, left=1in, right=1in]{geometry}

\usepackage{authblk} 

\newtheorem{theorem}{Theorem}[section]

\newtheorem{lemma}[theorem]{Lemma}
\newtheorem{corollary}[theorem]{Corollary}

\theoremstyle{definition}

\newtheorem{remark}[theorem]{Remark}

\providecommand{\Real}{\mathop{\rm Re}\nolimits}%
\providecommand{\Imag}{\mathop{\rm Im}\nolimits}%
\providecommand{\Res}{\mathop{\rm Res}}

\begin{document}



\title{The Lerch zeta function as a fractional derivative}

\author[1]{Arran Fernandez \thanks{Email: \texttt{af454@cam.ac.uk}}}
\affil[1]{{\small Department of Applied Mathematics \& Theoretical Physics, University of Cambridge, Cambridge, CB3 0WA, United Kingdom}}

\date{}

\maketitle

\begin{abstract}
We derive and prove a new formulation of the Lerch zeta function as a fractional derivative of an elementary function. We demonstrate how this formulation interacts very naturally with basic known properties of Lerch zeta, and use the functional equation to obtain a second formulation in terms of fractional derivatives.
\end{abstract}

\section{Introduction}
\label{intro}

Zeta functions are among the most important objects in the field of analytic number theory. The most famous of these is the \textit{Riemann zeta function}, defined by
\begin{equation}
\label{Riemann:defn}
\zeta(s)=\sum_{n=1}^{\infty}n^{-s}\quad\text{ for }\Real(s)>1,
\end{equation}
and by analytic continuation for all $s\in\mathbb{C}$. This function has been the subject of intense study for nearly two hundred years, mostly due to its connection with the distribution of prime numbers \cite{Edwards,Ivic,Titchmarsh}. It can be generalised in a number of directions: for example, the \textit{Dirichlet $L$-functions} are number-theoretical generalisations depending on both the complex variable $s$ and also a Dirichlet character modulo some base $d$, while the \textit{Hurwitz zeta function} and \textit{Lerch zeta function} are analytic generalisations depending on two or three independent complex variables. Specifically, the Hurwitz zeta function is defined by
\begin{equation}
\label{Hurwitz:defn}
\zeta(x,s)=\sum_{n=0}^{\infty}(n+x)^{-s}\quad\text{ for }\Real(s)>1,\Real(x)>0,
\end{equation}
and by analytic continuation for all $s\in\mathbb{C}$, while the Lerch zeta function is defined by
\begin{equation}
\label{Lerch:defn}
L(t,x,s)=\sum_{n=0}^{\infty}(n+x)^{-s}e^{2\pi itn}\quad\text{ for }\Real(s)>1,\Real(x)>0,\Imag(t)\geq0,
\end{equation}
and by analytic continuation for $(t,x,s)$ in larger domains \cite{LagariasII}, extending to a universal cover of the manifold $\mathbb{C} \backslash\mathbb{Z}\times\mathbb{C}\backslash\mathbb{Z}_0^-\times\mathbb{C}$. (We note that $t$ here is not the imaginary part of $s$, which it has sometimes \cite{Edwards,Titchmarsh} been used to denote, but an entirely independent variable.) It is clear that the Riemann, Hurwitz, and Lerch zeta functions are related by the following identities: \[\zeta(s)=\zeta(1,s);\quad\zeta(x,s)=L(0,x,s).\]

Many of the techniques used for analysing the Riemann zeta function and Dirichlet $L$-functions, such as the Euler product formula, have no general analogues for the Hurwitz or Lerch zeta functions. This is because the latter functions have a less direct connection to number theory, and are more readily studied using analytic methods. Indeed, many important facts about the Riemann zeta function do have analogues in the Hurwitz and Lerch cases \cite{Rane,GarunkstisI,GarunkstisII}, which are even proved in some cases by analogous methods. And analysing the Hurwitz and Lerch zeta functions can still be significant for number theory, purely because they include the Riemann zeta function as a special case.

In the current work, we shall be using the theory of \textit{fractional calculus}: derivatives and integrals to non-integer orders. It is possible to define the $n$th derivative of a function not just for $n\in\mathbb{N}$ but for any $n\in\mathbb{R}$ or even $n\in\mathbb{C}$. This field of study has a long history, stretching back to Hardy, Littlewood, Riemann, and even Leibniz, but only in recent decades has it begun to expand more rapidly. Much of this expansion is due to applications of fractional models being discovered throughout many areas of science, including chaos theory \cite{Hilfer,Petras}, bioengineering \cite{Magin}, stochastic processes \cite{Meerschaert}, and control theory \cite{Baleanu}.

Fractional-order derivatives and integrals can be defined in a number of ways, from the classical Riemann--Liouville and Caputo formulae \cite{Miller,Samko} to more recent variants such as the Caputo--Fabrizio, Atangana--Baleanu, and other models \cite{Caputo,Atangana,Jarad,Katugampola}. Here, we shall be using the basic Riemann--Liouville model, in which the fractional integral is defined by
\begin{equation}
\label{RLint:defn}
\prescript{}{c}D_{t}^{-\alpha}f(t)=\tfrac{1}{\Gamma(\alpha)}\int_c^t(t-u)^{\alpha-1}f(u)\,\mathrm{d}u\quad\text{ for }\Real(\alpha)>0,
\end{equation}
and the fractional derivative is defined by
\begin{equation}
\label{RLderiv:defn}
\prescript{}{c}D_{t}^{\alpha}f(t)=\frac{\mathrm{d}^n}{\mathrm{d}t^n}\Big(\prescript{}{c}D_{t}^{\alpha-n}f(t)\Big), \;n:=\lfloor\Real(\alpha)\rfloor+1,\quad\text{ for }\Real(\alpha)\geq0.
\end{equation}
In both \eqref{RLint:defn} and \eqref{RLderiv:defn}, the quantity $c$ is a complex constant, which can be thought of as a constant of integration. In most applications of Riemann--Liouville fractional calculus, $c$ is taken to be either $0$ or $-\infty$. The term \textit{differintegral} is used in fractional calculus to cover both derivatives and integrals, which in certain models can both be expressed by a unified formula.

When $t$ is a complex variable, the issue of branches and contours arises, since the term $(t-u)^{\alpha-1}$ appearing in \eqref{RLint:defn} is in general a multi-valued function. We usually take the contour of integration to be the straight line-segment from $c$ to $t$ in the complex $u$-plane, so that the argument of $t-u$ is fixed as $u$ varies. In the present work, we shall be using $c=-\infty$, so the contour of integration is a horizontal ray extending to the left from $t$, and we assume $\arg(t-u)=0$ so that the integrand of \eqref{RLint:defn} is a real multiple of $f(u)$.

There are various ways of motivating the definitions \eqref{RLint:defn} and \eqref{RLderiv:defn}. For example, the integral formula \eqref{RLint:defn} is a natural generalisation of Cauchy's formula for repeated integrals, or of Cauchy's integral formula in complex analysis, while the derivative formula \eqref{RLderiv:defn} arises naturally from consideration of semigroup properties and is also, for holomorphic functions $f$, the analytic continuation of \eqref{RLint:defn}.

We provide further motivation for the Riemann--Liouville formula by demonstrating that it works as expected for a few elementary functions $f$, in the following two lemmas.

\begin{lemma}
\label{RL:power}
The Riemann--Liouville differintegral of a power function, with constant of differintegration $c=0$, is given by
\begin{equation}
\label{RL:power:eqn}
\prescript{}{0}D_{t}^{\alpha}\left(t^{\beta}\right)=\frac{\Gamma(\beta+1)}{\Gamma(\beta-\alpha+1)}t^{\beta-\alpha},
\end{equation}
for $\alpha,\beta\in\mathbb{C}$ with $\Real(\beta)>-1$.
\end{lemma}

\begin{proof}
This follows directly from the definition of the beta function; the details may be found in \cite{Miller}.
\end{proof}

\begin{lemma}
\label{RL:exp}
The Riemann--Liouville differintegral of an exponential function, with constant of differintegration $c=-\infty$, is given by
\begin{equation}
\label{RL:exp:eqn}
\prescript{}{-\infty}D_{t}^{\alpha}\left(e^{kt}\right)=k^{\alpha}e^{kt},
\end{equation}
for $\alpha,k\in\mathbb{C}$ with $k\not\in\mathbb{R}^-_0$, where complex power functions are defined by the principal branch with arguments in the interval $(-\pi,\pi)$.
\end{lemma}

\begin{proof}[Proof (based on \cite{Samko})]
This follows from the definition of the gamma function, but care must be taken over the complex substitution in the integral. Note first that it will suffice to prove the result for $\Real(\alpha)<0$, since it will then follow for $\Real(\alpha)\geq0$ using the definition \eqref{RLderiv:defn}. Thus we assume $\Real(\alpha)<0$ and $k\not\in\mathbb{R}^-_0$, and use the definition \eqref{RLint:defn}: \[\prescript{}{-\infty}D_{t}^{\alpha}\left(e^{kt}\right)=\tfrac{1}{\Gamma(-\alpha)}\int_{-\infty}^t(t-u)^{-\alpha-1}e^{ku}\,\mathrm{d}u.\] Substituting $v=kt-ku$ yields
\begin{align*}
\prescript{}{-\infty}D_{t}^{\alpha}\left(e^{kt}\right)&=\tfrac{1}{\Gamma(-\alpha)}\int_{\infty}^0\left(\tfrac{v}{k}\right)^{-\alpha-1}e^{kt-v}\big(\tfrac{1}{-k}\big)\,\mathrm{d}v \\
&=\tfrac{1}{\Gamma(-\alpha)}e^{kt}\left(\tfrac{1}{k}\right)^{-\alpha}\int^{\infty}_0v^{-\alpha-1}e^{-v}\,\mathrm{d}v \\
&=e^{kt}\left(\tfrac{1}{k}\right)^{-\alpha}=k^{\alpha}e^{kt},
\end{align*}
where for the last step we used the fact that $k$ is not on the critical branch cut and therefore $k$ and $\frac{1}{k}$ both have arguments in $(-\pi,\pi)$.
\end{proof}

Another way of motivating the Riemann--Liouville definition is to note that it behaves exactly as expected with respect to Fourier and Laplace transforms. It is well known that standard differentiation and integration of a function correspond to multiplication of its Fourier transform by power functions. It turns out \cite{Miller,Samko} that the same is true for Riemann--Liouville fractional differintegrals:
\begin{align*}
\mathcal{F}\left[\prescript{}{-\infty}D_t^{\alpha}f(t)\right]&=(i\omega)^{\alpha}\mathcal{F}[f(t)]; \\
\mathcal{L}\left[\prescript{}{0}D_t^{\alpha}f(t)\right]&=\omega^{\alpha}\mathcal{L}[f(t)].
\end{align*}
As we are only mentioning this identities for the sake of motivation, and they are not relevant to the main arguments of this paper, we omit the proofs, and refer the reader to \cite[\S7]{Samko} for rigorous statements of the results with all required assumptions.

The following result concerning fractional integration of series will be used later on in the proof of the main result.

\begin{lemma}
\label{RL:series}
If the series $f(t)=\sum_{n=1}^{\infty}f_n(t)$ is uniformly convergent on a complex disc $|t-c|\leq R$ with $c-R\not\in\mathbb{R}^+_0$, and the constants $\delta$, $\alpha$ satisfy $\delta>0$, $\Real(\alpha)<0$, and \[\left[\sum_{n=N+1}^{\infty}f_n(t)\right]t^{\delta-\alpha}\rightarrow0\text{ as }N\rightarrow\infty\] uniformly on the ray from $c-R$ to negative infinity, then we have
\begin{equation*}
\prescript{}{-\infty}D_{t}^{\alpha}f(t)=\sum_{n=1}^{\infty}\prescript{}{-\infty}D_{t}^{\alpha}f_n(t)
\end{equation*}
for $|t-c|\leq R$, and the series of fractional integrals is locally uniformly convergent.
\end{lemma}

\begin{proof}
This result is established by the proof of \cite[Theorem IX]{Keiper}. (In that proof, it was assumed that $c$ is real, but this was only for convenience -- the same argument works for complex $c$ provided that $c-R\not\in\mathbb{R}^+_0$.)
\end{proof}

Despite the increasing usefulness and applications of fractional calculus, it has so far been largely neglected as a tool in analytic number theory. The idea of bringing fractional calculus and analytic number theory together was born in the work of Keiper, who in his 1975 MSc thesis \cite{Keiper} established a formula for the Riemann zeta function as a Riemann--Liouville fractional derivative. It has only been revived very recently, in the work of Guariglia et al \cite{GuarigliaI,GuarigliaII,GuarigliaIII} and also Srivastava et al \cite{SrivastavaI,SrivastavaII,SrivastavaIII} -- but the Guariglia papers use a different model of fractional calculus, namely a recent variant due to Ortigueira of the Caputo model, while the Srivastava papers only consider fractional expressions for generalisations of the Lerch zeta function in terms of each other, not in terms of elementary functions.

Here, we establish a new relationship between fractional calculus and zeta functions, by writing the Lerch zeta function as a fractional derivative of a much simpler function. We use only the classical Riemann--Liouville model of fractional calculus, without the complications introduced by newer models. Furthermore, we must necessarily use the Lerch zeta function rather than the Hurwitz or Riemann zeta functions, since the third parameter $t$ in $L(t,x,s)$ plays a vital role in our derivation. This may explain why our formula has not been discovered before. Of course, it does yield a new expression for the Riemann zeta function too, simply by setting $x=1$ and $t=0$.

This paper is organised as follows. In section 2, we derive the main result, justify its naturality by several remarks to verify various aspects of it, and use it to deduce further formulae linking zeta functions with fractional differintegrals. In section 3, we comment on possible applications and extensions of our results.

\section{The main results}

The crux of this work is the following theorem expressing the Lerch zeta function as a fractional differintegral.

\begin{theorem}
\label{result}
The Lerch zeta function can be written as
\begin{equation}
\label{result:eqn}
L(t,x,s)=(2\pi)^s\exp\left[i\pi(\tfrac{s}{2}-2tx)\right]\prescript{}{-\infty}D_{t}^{-s}\left(\frac{e^{2\pi itx}}{1-e^{2\pi it}}\right)
\end{equation}
for any complex numbers $s,x,t$ satisfying $\Imag(t)>0$ and $x\not\in(-\infty,0]$.
\end{theorem}

\begin{proof}
We start from the definition \eqref{Lerch:defn} of the Lerch zeta function, and use the result of Lemma \ref{RL:exp} to rewrite the summand as a fractional differintegral:
\begin{align}
\nonumber L(t,x,s)&=\sum_{n=0}^{\infty}(n+x)^{-s}e^{2\pi itn}=(2\pi i)^se^{-2\pi itx}\sum_{n=0}^{\infty}(2\pi i)^{-s}(n+x)^{-s}e^{2\pi it(n+x)} \\
\label{Lerch:first} &=(2\pi i)^se^{-2\pi itx}\sum_{n=0}^{\infty}\prescript{}{-\infty}D_{t}^{-s}\left(e^{2\pi it(n+x)}\right).
\end{align}
So far our argument is valid for all $t,x,s\in\mathbb{C}$ such that $\Real(s)>1$, $\Real(x)>0$, and $\Imag(t)\geq0$. These conditions come from the definition \eqref{Lerch:defn}; the extra condition that $2\pi i(n+x)\not\in\mathbb{R}^-_0$, required by Lemma \ref{RL:exp}, is automatically satisfied for all $n\geq0$ due to the condition we already have on $x$. Note that since $s$ has positive real part, the fractional operator appearing in \eqref{Lerch:first} is an integral and not a derivative.

The next consideration is whether or not the summation and fractional integration operators in \eqref{Lerch:first} can be swapped. 
For any $\epsilon>0$, the series \[\sum_{n=0}^{\infty}e^{2\pi it(n+x)}\] converges uniformly on the closed region $\Imag(t)\geq\epsilon$ of the upper half $t$-plane, and indeed \[\left(\sum_{n=N+1}^{\infty}e^{2\pi it(n+x)}\right)t^{\delta-s}\rightarrow0\text{ as }N\rightarrow\infty\] uniformly on this region for any fixed $\delta<1$. So, under the slightly strengthened condition $\Imag(t)>0$, it follows from Lemma \ref{RL:series} that the series of fractional integrals also converges locally uniformly and
\begin{equation*}
\prescript{}{-\infty}D_{t}^{-s}\left(\sum_{n=0}^{\infty}e^{2\pi it(n+x)}\right)=\sum_{n=0}^{\infty}\prescript{}{-\infty}D_{t}^{-s}\left(e^{2\pi it(n+x)}\right).
\end{equation*}
Substituting this identity into the expression \eqref{Lerch:first} yields:
\begin{align*}
L(t,x,s)&=(2\pi i)^se^{-2\pi itx}\prescript{}{-\infty}D_{t}^{-s}\left(\sum_{n=0}^{\infty}e^{2\pi it(n+x)}\right) \\
&=(2\pi i)^se^{-2\pi itx}\prescript{}{-\infty}D_{t}^{-s}\left(e^{2\pi itx}\sum_{n=0}^{\infty}\left(e^{2\pi it}\right)^n\right) \\
&=(2\pi i)^se^{-2\pi itx}\prescript{}{-\infty}D_{t}^{-s}\left(\frac{e^{2\pi itx}}{1-e^{2\pi it}}\right) \\
&=(2\pi)^s\exp\left[i\pi(\tfrac{s}{2}-2tx)\right]\prescript{}{-\infty}D_{t}^{-s}\left(\frac{e^{2\pi itx}}{1-e^{2\pi it}}\right),
\end{align*}
as required.

We have now proved the main result \eqref{result:eqn} under the following assumptions:
\begin{equation*}
\Real(s)>1,\quad\Real(x)>0,\quad\Imag(t)>0.
\end{equation*}
By analytic continuation, these assumptions can be relaxed to any $t,x,s\in\mathbb{C}$ such that both sides of \eqref{result:eqn} are still holomorphic. We know from \cite[Theorem 2.3]{LagariasII} that the left-hand side $L(t,x,s)$ can be extended to a holomorphic function on the domain
\[\{(t,x,s)\in\mathbb{C}\times\mathbb{C}\times\mathbb{C}:\Imag(t)>0,x\not\in(-\infty,0]\},\]
this domain being embeddable into the universal cover of $(\mathbb{C}\backslash\mathbb{Z})\times(\mathbb{C}\backslash\mathbb{Z}^-_0)\times\mathbb{C}$.

The right-hand side of \eqref{result:eqn} is clearly going to be holomorphic in $x$ wherever it is well-defined, and ditto in $s$ by \cite[\S2.4]{Samko}. It is well-defined and holomorphic in $t$ provided that the fractional differintegral is well-defined and holomorphic in $t$.

For $\Real(s)>0$, this differintegral can be written as
\begin{equation}
\label{result:FI}
\tfrac{1}{\Gamma(s)}\int_{-\infty}^t(t-u)^{s-1}\frac{e^{2\pi iux}}{1-e^{2\pi iu}}\,\mathrm{d}u.
\end{equation}
The integrand here is holomorphic in $u\in\mathbb{C}\backslash\mathbb{Z}$, since we are assuming the contour of integration to be horizontal in the complex plane. Thus the whole expression is well-defined and holomorphic for any $t,x,s$ such that $\Imag(t)>0$ and the integral converges at both endpoints.

Near $u=t$, the exponential-fraction part of the integrand is constant, so the integral behaves like $(t-u)^s$, which converges since we have assumed $\Real(s)>0$.

Near $u=-\infty$, the exponential denominator is bounded (since we have $\Imag(u)>0$), the numerator has exponential decay provided that $\Imag(x)<0$, and the $(t-u)^{s-1}$ term has only polynomial growth.

Thus the expression \eqref{result:FI} is well-defined and holomorphic in all three variables provided that $\Real(s)>0$, $\Imag(x)<0$, and $\Imag(t)>0$.

We can extend the region of validity to cover $x\in\mathbb{R}\setminus\mathbb{Z}^-_0$ too, given an extra restriction on $s$. The series \[\sum_{n=0}^{\infty}e^{2\pi iu(x+n)}=\frac{e^{2\pi iux}}{1-e^{2\pi iu}}\] is uniformly convergent, since $u$ has a fixed positive imaginary part. Therefore \eqref{result:FI} can be rewritten, regardless of $x$, in the form of the series \[\tfrac{1}{\Gamma(s)}\sum_{n=0}^{\infty}\int_{-\infty}^t(t-u)^{s-1}e^{2\pi iu(x+n)}\,\mathrm{d}u,\] whose integral summand is well-defined for $x\in\mathbb{R}\backslash\mathbb{Z}^-_0$ provided that $\Real(s)\leq1$.

Given the definition \eqref{RLderiv:defn} of fractional derivatives, the $\Real(s)>0$ requirement can be eliminated immediately.

So the main result \eqref{result:eqn} is now proved under the following assumption:
\begin{equation*}
\Imag(x)<0,\Imag(t)>0\quad\text{ or }\quad \Real(s)\leq1,x\in\mathbb{R}\setminus\mathbb{Z}^-_0,\Imag(t)>0.
\end{equation*}
But we already know that \eqref{result:eqn} is also valid for $\Real(s)>1,\Real(x)>0,\Imag(t)>0$. Thus, by taking unions of domains, we can say that it is always valid for
\begin{equation*}
\Imag(x)<0 \text{ or } x\in\mathbb{R}^+,\quad\quad\Imag(t)>0.
\end{equation*}

Finally, it is clear from the definition \eqref{Lerch:defn} that the Lerch zeta function satisfies the following basic functional equation:
\begin{equation}
\label{Lerch:conj}
\overline{L(t,x,s)}=L(-\bar{t},\bar{x},\bar{s}).
\end{equation}
The condition $\Imag(t)>0$ is preserved by mapping $t$ to $-\bar{t}$, but if $\Imag(x)\leq0$, then $\Imag(\bar{x})\geq0$. Thus, if \eqref{result:eqn} is known to be valid for the lower half plane part of $\mathbb{C}\backslash\mathbb{R}^-_0$, then by taking complex conjugates it follows that it is also valid for the upper half plane part of $\mathbb{C}\backslash\mathbb{R}^-_0$, and therefore for all $x\in\mathbb{C}\backslash\mathbb{R}^-_0$.
\end{proof}

\begin{remark}
Note that unlike previous results on the fractional calculus of zeta functions \cite{Keiper,GuarigliaI}, our formula depends crucially on using the Lerch zeta function rather than the Riemann or Hurwitz zeta functions. The third parameter $t$ -- i.e. the one which appears in the Lerch function $L(t,x,s)$ but not the Riemann or Hurwitz functions -- is a fundamental part of our result \eqref{result:eqn}: we could not have achieved analogous results for $\zeta(s)$ or $\zeta(x,s)$ without first introducing this extra parameter in order to differentiate with respect to it.

It is, however, possible to obtain a formula for the Riemann zeta function as a corollary of Theorem \ref{result}, as follows.
\end{remark}

\begin{corollary}
The Riemann zeta function can be written as
\begin{equation}
\label{result:Riemann1}
\zeta(s)=\frac{(2\pi i)^s}{2^{1-s}-1}\prescript{}{-\infty}D_{t}^{-s}\left(\frac{1}{e^{-2\pi it}-1}\right)\Bigg|_{t=\frac{1}{2}}
\end{equation}
for any $s\in\mathbb{C}$, or alternatively as
\begin{equation}
\label{result:Riemann2}
\zeta(s)=(2\pi i)^s\lim_{t\rightarrow0}\left(\prescript{}{-\infty}D_{t}^{-s}\left(\frac{1}{e^{-2\pi it}-1}\right)\right)
\end{equation}
for $\Real(s)>1$.
\end{corollary}

\begin{proof}
The first identity \eqref{result:Riemann1} follows by letting $t=\frac{1}{2}$ in \eqref{result:eqn} and noting the fact that \[L(\frac{1}{2},1,s)=(1-2^{1-s})\zeta(s).\]

The second identity \eqref{result:Riemann2} follows by letting $t\rightarrow0$ in \eqref{result:eqn} and recalling the series definitions \eqref{Riemann:defn},\eqref{Lerch:defn}. We note that \eqref{result:Riemann2} does not hold in general, because the limit as $t\rightarrow0$ of the Lerch function does not always exist \cite{Oberhettinger}.
\end{proof}

\begin{remark}
We verify that our new formula satisfies the complex conjugation relation \eqref{Lerch:conj} for the Lerch zeta function. Using the right-hand side of \eqref{result:eqn} as the definition of $L(t,x,s)$, we get:
\begin{align*}
\overline{L(t,x,s)}&=(2\pi)^{\bar{s}}\exp\left[-i\pi(\tfrac{\bar{s}}{2}-2\bar{t}\bar{x})\right]\overline{\prescript{}{-\infty}D_{u=t}^{-s}\left(\frac{e^{2\pi iux}}{1-e^{2\pi iu}}\right)} \\
L(-\bar{t},\bar{x},\bar{s})&=(2\pi)^{\bar{s}}\exp\left[i\pi(\tfrac{\bar{s}}{2}+2\bar{t}\bar{x})\right]\prescript{}{-\infty}D_{u=-\bar{t}}^{-\bar{s}}\left(\frac{e^{2\pi iu\bar{x}}}{1-e^{2\pi iu}}\right)
\end{align*}
(We use the notation $D^{\alpha}_{u=t}f(u)$ instead of $D^{\alpha}_tf(t)$ in order to avoid confusion in the case where $t$ is replaced by $-\bar{t}$.) Thus, to verify \eqref{Lerch:conj} it will be sufficient to show that
\begin{equation*}
\overline{\prescript{}{-\infty}D_{u=t}^{-s}\left(\frac{e^{2\pi iux}}{1-e^{2\pi iu}}\right)}=e^{i\pi s}\prescript{}{-\infty}D_{u=-\bar{t}}^{-\bar{s}}\left(\frac{e^{2\pi iu\bar{x}}}{1-e^{2\pi iu}}\right),
\end{equation*}
or in other words, assuming $\Real(s)>0$, \[\tfrac{1}{\Gamma(\bar{s})}\overline{\int_{-\infty}^t(t-u)^{s-1}\frac{e^{2\pi iux}}{1-e^{2\pi iu}}\,\mathrm{d}u}=e^{i\pi s}\tfrac{1}{\Gamma(\bar{s})}\int_{-\infty}^{-\bar{t}}(-\bar{t}-u)^{\bar{s}-1}\frac{e^{2\pi iu\bar{x}}}{1-e^{2\pi iu}}\,\mathrm{d}u.\] Writing $t=a+bi$ and $-\bar{t}=-a+bi$ and $u=r+bi$, this becomes \[\int_{-\infty}^a(a-r)^{\bar{s}-1}\frac{e^{-2\pi i\bar{u}\bar{x}}}{1-e^{-2\pi i\bar{u}}}\,\mathrm{d}r=e^{i\pi s}\int_{-\infty}^{-a}(-a-r)^{\bar{s}-1}\frac{e^{2\pi iu\bar{x}}}{1-e^{2\pi iu}}\,\mathrm{d}r.\] Since $b>0$, both denominators can be expanded as series, so it is sufficient to prove that \[\int_{-\infty}^a(a-r)^{\bar{s}-1}e^{-2\pi i\bar{u}(\bar{x}+n)}\,\mathrm{d}r=e^{i\pi s}\int_{-\infty}^{-a}(-a-r)^{\bar{s}-1}e^{2\pi iu(\bar{x}+n)}\,\mathrm{d}r\] for all $n\in\mathbb{Z}^+_0$. Making a linear substitution and factoring out constant terms, this reduces to \[\int_0^{\infty}p^{\bar{s}-1}e^{2\pi ip(\bar{x}+n)}\,\mathrm{d}p=e^{i\pi s}\int_0^{\infty}p^{\bar{s}-1}e^{-2\pi ip(\bar{x}+n)}\,\mathrm{d}p,\] or equivalently \[\int_0^{\infty}p^{\bar{s}-1}e^{2\pi ip(\bar{x}+n)}\,\mathrm{d}p=-\int_{-\infty}^0p^{\bar{s}-1}e^{2\pi ip(\bar{x}+n)}\,\mathrm{d}p,\] where the integral along the negative real axis is assumed to be with argument $+\pi$. And by Jordan's lemma, closing the real contour in the upper half plane gives \[\int_{-\infty}^{\infty}p^{\bar{s}-1}e^{2\pi ip(\bar{x}+n)}\,\mathrm{d}p=0\] for all $n\geq0$, provided that $x\in\mathbb{R}^+$ and $\Real(s)<1$.

So we have re-verified the identity \eqref{Lerch:conj} under the assumptions $0<\Real(s)<1,x\in\mathbb{R}^+,\Imag(t)>0$. This acts as a confirmation of the correctness of our result.
\end{remark}

The result of Theorem \ref{result} is an expression for the Lerch zeta function as the product of a fractional differintegral and a simple explicit term. We now demonstrate how this explicit term arises naturally from consideration of the Lerch zeta function and its properties, and thence derive a second formula for the Lerch zeta function in terms of fractional differintegrals.

\begin{remark}
It is known \cite{Apostol,LagariasI,Lerch} that for $s,t\in\mathbb{C}$ with $\Imag(t)>0$ and $x\in(0,1)$, or with $t,x\in(0,1)$, the Lerch zeta function satisfies the following functional equation:
\begin{multline}
\label{Lerch:fnleqn}
L(t,x,1-s)=\frac{\Gamma(s)}{(2\pi)^s}\Big(\exp\left[i\pi(\tfrac{s}{2}-2tx)\right]L(-x,t,s) \\ +\exp\left[-i\pi(\tfrac{s}{2}-2x(1-t)\right]L(x,1-t,s)\Big)
\end{multline}
Thus, we observe that the exponential multiplier term $\exp\left[i\pi(\tfrac{s}{2}-2tx)\right]$ seen in \eqref{result:eqn} is already known to arise from essential properties of the Lerch zeta function. This demonstrates the naturality of the result of Theorem \ref{result}.
\end{remark}

\begin{theorem}
\label{result2}
The Lerch zeta function can be written as
\begin{equation}
\label{result2:eqn}
L(t,x,1-s)=\Gamma(s)e^{i\pi s}\prescript{}{-\infty}D_{u}^{-s}\left(\frac{e^{2\pi itu}}{1-e^{-2\pi iu}}\right)\bigg|_{u=-x}-\Gamma(s)\prescript{}{-\infty}D_{x}^{-s}\left(\frac{e^{-2\pi itx}}{1-e^{2\pi ix}}\right)
\end{equation}
where $s,x,t$ are any complex numbers satisfying $\Imag(t)>0$ and $x\in(0,1)$.
\end{theorem}

\begin{proof}
In order to use the identity \eqref{Lerch:fnleqn} together with the new expression \eqref{result:eqn}, we will need to show that \eqref{result:eqn} can be extended from $\Imag(t)>0$ to the line $t\in\mathbb{R}\backslash\mathbb{Z}$. This can be shown by continuity, provided that we choose the right contour for the integration inherent in the fractional differintegral. When $t\in\mathbb{R}\backslash\mathbb{Z}$, the straight ray from $t$ to $-\infty$ contains infinitely many poles of the function $\frac{e^{2\pi itx}}{1-e^{2\pi it}}$, so the integral must be defined as a limit:
\begin{equation}
\label{reallimit}
\prescript{}{-\infty}D_{t}^{-s}\left(\frac{e^{2\pi itx}}{1-e^{2\pi it}}\right)=\lim_{\epsilon\rightarrow0^+}\left[\tfrac{1}{\Gamma(s)}\int_{-\infty}^{t'}(t'-u)^{s-1}\frac{e^{2\pi iux}}{1-e^{2\pi iu}}\,\mathrm{d}u\right]_{t'=t+i\epsilon},\quad t\in\mathbb{R}.
\end{equation}
With this definition, it is clear by continuity that \eqref{result:eqn} still holds for all $t$ with $\Imag(t)\geq0,t\not\in\mathbb{Z}$.

Now we can start from the functional equation \eqref{Lerch:fnleqn} and substitute \eqref{result:eqn} for the two Lerch functions on the right-hand side. For simplicity, we shall drop the left-subscript $-\infty$ on the fractional operators, since they all use the same constant of differintegration. We also use the notation $D^{\alpha}_{u=x}f(u)$ instead of $D^{\alpha}_xf(x)$, just to avoid confusion in the case where $x$ is replaced by $-x$.
\begin{align*}
\begin{split}
L(t,x,1-s)&=\frac{\Gamma(s)}{(2\pi)^s}\Big(\exp\left[i\pi(\tfrac{s}{2}-2tx)\right]L(-x,t,s) \\
&\hspace{4cm}+\exp\left[-i\pi(\tfrac{s}{2}-2x(1-t)\right]L(x,1-t,s)\Big)
\end{split} \\
\begin{split}
&\hspace{-1cm}=\frac{\Gamma(s)}{(2\pi)^s}\bigg(\exp\left[i\pi(\tfrac{s}{2}-2tx)\right](2\pi)^s\exp\left[i\pi(\tfrac{s}{2}+2tx)\right]D_{u=-x}^{-s}\left(\frac{e^{2\pi itu}}{1-e^{2\pi iu}}\right) \\
&\hspace{-0.5cm}+\exp\left[-i\pi(\tfrac{s}{2}-2x(1-t)\right](2\pi)^s\exp\left[i\pi(\tfrac{s}{2}-2(1-t)x)\right]D_{u=x}^{-s}\left(\frac{e^{2\pi i(1-t)u}}{1-e^{2\pi iu}}\right)\bigg)
\end{split} \\
&\hspace{-1cm}=\Gamma(s)\bigg(\exp\left[i\pi s\right]D_{u=-x}^{-s}\left(\frac{e^{2\pi itu}}{1-e^{2\pi iu}}\right)+D_{u=x}^{-s}\left(\frac{e^{2\pi i(1-t)u}}{1-e^{2\pi iu}}\right)\bigg) \\
&\hspace{-1cm}=\Gamma(s)\bigg(e^{i\pi s}D_{u=-x}^{-s}\left(\frac{e^{2\pi itu}}{1-e^{2\pi iu}}\right)-D_{u=x}^{-s}\left(\frac{e^{-2\pi itu}}{1-e^{-2\pi iu}}\right)\bigg).
\end{align*}
And the required result follows.
\end{proof}

\begin{remark}
The results of Theorems \ref{result} and \ref{result2} can be used to provide a new elementary proof of the functional equation \eqref{Lerch:fnleqn}.

In the proof of Theorem \ref{result2}, we used the new expression \eqref{result:eqn} for the Lerch zeta function to reduce the right-hand side of the functional equation \eqref{Lerch:fnleqn} to an expression in terms of two fractional differintegrals. If we can rewrite this expression using elementary methods as simply $L(t,x,1-s)$, then we have rederived the functional equation using fractional calculus.

Therefore, we start from the right-hand side of \eqref{result2:eqn} and proceed as follows:
\begin{align*}
(\text{RHS of \eqref{result2:eqn}})&=e^{i\pi s}\int_{-\infty}^{-x}(-x-u)^{s-1}\frac{e^{2\pi itu}}{1-e^{-2\pi iu}}\,\mathrm{d}u-\int_{-\infty}^x(x-u)^{s-1}\frac{e^{-2\pi itu}}{1-e^{2\pi iu}}\,\mathrm{d}u \\
&=e^{i\pi s}\int^{\infty}_{x}(-x+u)^{s-1}\frac{e^{-2\pi itu}}{1-e^{2\pi iu}}\,\mathrm{d}u-\int_{-\infty}^x(x-u)^{s-1}\frac{e^{-2\pi itu}}{1-e^{2\pi iu}}\,\mathrm{d}u \\
&=-\int_{-\infty}^{\infty}(x-u)^{s-1}\frac{e^{-2\pi itu}}{1-e^{2\pi iu}}\,\mathrm{d}u,
\end{align*}
where the contour of integration from $-\infty$ to $+\infty$ crosses the real axis at $x$, passing above all the poles to the left of $x$ and below all the poles to the right of $x$. This choice of contour follows from the definition given by \eqref{reallimit}.

By Jordan's lemma, for $t\in\mathbb{R}$ and $\Real(s)<1$, the contour can be closed in the lower half plane. Then the residue theorem yields
\begin{align*}
(\text{RHS of \eqref{result2:eqn}})&=2\pi i\sum_{n=0}^{\infty}\Res_{u=-n}\left((x-u)^{s-1}\frac{e^{-2\pi itu}}{1-e^{2\pi iu}}\right) \\
&=2\pi i\sum_{n=0}^{\infty}(x+n)^{s-1}\frac{e^{2\pi itn}}{2\pi i} \\
&=\sum_{n=0}^{\infty}(x+n)^{s-1}e^{2\pi itn}=L(t,x,1-s),
\end{align*}
as required. Thus we have proved the functional equation \eqref{Lerch:fnleqn} in the case where $0<x<1,0<t<1,\Real(s)<1$.
\end{remark}

\section{Conclusions}

In this paper, we have forged a new connection between fractional calculus and the theory of zeta functions. This connection is different from others that have previously been discovered: it was found by using the Lerch zeta function, a significant generalisation of the more commonly seen Riemann and Hurwitz zeta functions, and it enables all of these zeta functions to be expressed as fractional derivatives of very basic functions.

We have also demonstrated the usefulness of our result by indicating its natural interplay with fundamental properties of zeta functions, and how it can even be used to provide new proofs of some of these properties.

Any new formula for zeta functions is potentially useful, as it gives a new angle of attack in the ceaseless attempts to establish important properties of such functions. It is especially important to establish more links between fractional calculus and analytic number theory, in order to increase the probability that all the machinery of one field can be brought to bear on the problems of the other.

The formulae proved in this paper could be just the start of a whole new project bringing together two distinct fields of study. For example, basic theorems of fractional calculus, such as analogues of the product rule and chain rule \cite{OslerI,OslerII,OslerIII}, may now be usable to generate significant new expressions for zeta functions. Creating new links between different areas is always an opportunity, and this is surely no exception.

\subsection*{Acknowledgments}

The author is grateful to Professors Athanassios S. Fokas and Dumitru Baleanu for inspiring discussions and recommendations to the literature, and also to the anonymous reviewer for their very helpful remarks and suggestions.

\end{document}